\documentclass[preprint,review,10pt]{amsart}
\usepackage{amsmath,amssymb,amsfonts,xspace}
\newtheorem{theorem}{Theorem}[section]

\theoremstyle{definition}
\newtheorem{definition}[theorem]{Definition}
\newtheorem{example}[theorem]{Example}

\newtheorem{proposition}[theorem]{Proposition}
\newtheorem{corollary}[theorem]{Corollary}
\newtheorem{lem}[theorem]{Lemma}

\theoremstyle{remark}
\newtheorem{remark}[theorem]{Remark}

\numberwithin{equation}{section}

\begin{document}

\title{  $(u,v)$-absorbing (prime) hyperideals in commutative multiplicative hyperrings }

\author{Mahdi Anbarloei}
\address{Department of Mathematics, Faculty of Sciences,
Imam Khomeini International University, Qazvin, Iran.
}

\email{m.anbarloei@sci.ikiu.ac.ir}

%\author{}
%\address{}
%\email{}

%    General info

%\date{September  , 2013.}

%------------------------------------------------------------------------------
%%%%%%%%%%%%%%%%%%%%%%%%%%%%%%%%%%%%%%%%%%%%%%%%%%%%%%%%%%%%%%%%%%%%%%%%%%%%%%%%%%%%%%%%%%%%%%%%%%%%%%

%%%%%%%%%%%%%%%%%%%%%%%%%%%%%%%%%%%%%%%%%%%%%%%%%%%%%%%%%%%%%%%%%%%%%%%%%%%%%%%%%%%%%%%%%%%%%%%%%%%%%%%%%%%%%%%%%%%%%%%%%%%%%%%%%%%%%%%%%
\begin{abstract}
 In this paper, we will introduce the notion of $(u,v)$-absorbing hyperideals
in multiplicative hyperrings and we will show some properties of them.
Then we extend this concept to the notion  of $(u,v)$-absorbing prime hyperideals  and then we will give some results about  them.
\end{abstract}
%%%%%%%%%%%%%%%%%%%%%%%%%%%%%%%%%%%%%%%%%%%%%%%%%%%%%%%%%%%%%%%%%%%%%%%%%%%%%%%%%
\maketitle

\section{Introduction}
In 1934 \cite{marty}, the French mathematician F. Marty presented the notion of hyperstructures or multioperations in which  an operation yields a set of values rather than a single value.
 Many  authors have contributed to the advancement of this novel area of modern algebra \cite{f1,f2,f3,f4,f5,f7,f8,f9}. 
An importan class of the algebraic hyperstructures  called the multiplicative hyperring was defined by Rota in 1982 \cite{f14}. In this hyperstructure,  the multiplication is a hyperoperation however  the addition is an operation.  Multiplicative hyperrings are deeply demonstrated and characterized in  \cite{ameri5, ameri6,  anb4,   Kamali, Ghiasvand, Ghiasvand2, f16, ul}.  Dasgupta studied the prime and primary hyperideals in multiplicative hyperrings in \cite{das}. The notions of 1-absorbing  and $v$-absorbing hyperideals as  generalizations of the prime hyperideals were studied in \cite{Ghiasvand2} and \cite{anb4}. 
The idea of $(u,v)$-closed hyperideals in a multiplicative hyperring was proposed in \cite{anb7}. Our purpose  of this research work is to  introduce and study   the notion of $(u,v)$-absorbing  hyperideals  as an extension of the $n$-absorbing hyperideals in  a commutative multiplicative hyperring.   Several specific properties  are given to illustrate  the structures of the new notion.  Furthermore,  we  extend this concept to $(u,v)$-absorbing  prime hyperideals.

%%%%%%%%%%%%%%%%%%%%%%%%%%%%%
%%%%%%%%%%%%%%%%%%%%%%%%%%%%%%
%%%%%%%%%%%%%%%%%%%%%%%%%%%%%%%%%%
\section{ multiplicative hyperrings}
In this section, we  recall some background materials.   A hyperoperation $``\circ" $ on non-empty set $X$ is a mapping from $X \times X$ into $P^*(X)$ such that $P^*(X)$ is the family of all non-empty subsets of $X$. In this case,  $(X,\circ)$ is called hypergroupoid.  Let $X_1,X_2$ be two subsets of $X$ and $x \in X$, then $X_1 \circ X_2 =\cup_{x_1 \in X_1, x_2 \in X_2}x_1 \circ x_2,$ and $ X_1 \circ x=X_1 \circ \{x\}.$ This means that the hyperoperation $``\circ" $ on $X$ can be extended to  subsets of $X$. A hypergroupoid $(X, \circ)$ is called  a semihypergroup if $\cup_{a \in y \circ z}x \circ a=\cup_{b \in x \circ y} b \circ z $ for all $x,y,z \in X$ which means $\circ$ is associative. A semihypergroup $X$ is called a hypergroup if  $x \circ X=X=X\circ x$ for each  $x \in X$. A non-empty subset $Y$ of a semihypergroup $(X,\circ)$ refers to a subhypergroup if   $x \circ Y=Y=Y \circ x$ for each $x \in Y$  \cite{f10}. 

\begin{definition}
 \cite{f10} An algebraic structure $(A,+,\circ)$ refers to a  commutative multiplicative hyperring if 
 \begin{itemize}
\item[\rm(1)]~ $(A,+)$ is a commutative group; 
\item[\rm(2)]~ $(A,\circ)$ is a semihypergroup; 
\item[\rm(3)]~ $x\circ (y+z) \subseteq x\circ y+x\circ z$ and $(y+z)\circ x \subseteq y\circ x+z\circ x$ for every $x, y, z \in A; $
\item[\rm(4)]~  $x\circ (-y) = (-x)\circ y = -(x\circ y)$ for every $x, y \in A$; 
\item[\rm(5)]~  $x \circ y =y \circ x$ for every $x,y \in A$.
 \end{itemize}
 \end{definition}
If in (3), the equality holds then the multiplicative hyperring $A$ is called strongly distributive. 

Conseder the ring of integers $(\mathbb{Z},+,\cdot)$. For each subset $T \in P^\star(\mathbb{Z})$ with $\vert T\vert \geq 2$, there exists a multiplicative hyperring $(\mathbb{Z}_T,+,\circ)$ where $\mathbb{Z}_T=\mathbb{Z}$ and  $a \circ b =\{a.x.b\ \vert \ x \in T\}$ for all $a,b\in \mathbb{Z}_T$ \cite{das}.
 \begin{definition}\cite{ameri} 
%An element $e \in A$ refers to   a scalar identity element if $a= a\circ e$ for all $a \in A$ . Moreover,
An element $e \in A$ is considered as an identity element if $a \in a\circ e$ for all $a \in A$.
\end{definition}
Throughout this paper, $A$ denotes a commutative multiplicative hyperring with identity 1. 
 \begin{definition} \cite{ameri} 
 An element $x \in A$ is called unit, if there exists $y \in A$ such that $1 \in x \circ y$.  Denote the set of all unit elements in $A$ by $U(A)$.
 \end{definition}
 Moreover, a hyperring $A$ refers to a hyperfield if each non-zero element of $A$ is unit.
\begin{definition} \cite{f10} A non-empty subset $B$ of  $A$ is a  hyperideal  if 
 \begin{itemize}
\item[\rm(i)]~  $x - y \in B$ for all $x, y \in B$; 
\item[\rm(ii)]~ $r \circ x \subseteq B$ for all $x \in B$ and $r \in A$.
 \end{itemize}
 \end{definition}
%A hyperideal $A$ of $H$ is finitely generated if $A=\langle x_1,\cdots,x_n \rangle$ for $x_1,\cdots,x_n \in H$, namely, for any $x \in A$, there exist $a_1,\cdots,a_n \in H$ such that $x \in a_1 \circ x_1+\cdots+a_n \circ x_n$. A hyperideal $A$ of $H$ is principal if $A=\langle x \rangle$ for $x \in H$. If each hyperideal of $H$  is principal, then $H$ is a principal hyperideal hyperring. 
\begin{definition}
 \cite{das} A proper hyperideal $B$ in  $A$ refers to a prime hyperideal if $x \circ y \subseteq B$ for $x,y \in A$, then $x \in B$ or $y \in B$. 
 \end{definition} 
  The intersection of all prime hyperideals of $A$ containing a hyperideal $B$ is said to be the prime radical of $B$,  denoted by $rad (B)$. If the multiplicative hyperring $A$  has no prime hyperideal containing $B$, we define $rad(B)=A$. Assume that $\mathcal{C}$ is the class of all finite products of elements of $A$ that is $\mathcal{C} = \{a_1 \circ a_2 \circ \cdots \circ a_n \ \vert \ a_i \in A, n \in \mathbb{N}\} \subseteq P^{\ast }(A)$ and  $B$ is a hyperideal of $A$. $B$ refers to a $\mathcal{C}$-hyperideal of $A$ if for each $C \in \mathcal{C}$ and $ B \cap C \neq \varnothing $ imply $C \subseteq B$. Notice that 
 $\{a \in A \ \vert \  a^n \subseteq B \ \text{for some} \ n \in \mathbb{N}\} \subseteq rad(A)$. The equality holds if $B$ is a $\mathcal{C}$-hyperideal of $A$ (see Proposition 3.2 in \cite{das}). Moreover, a hyperideal $B$ of $A$ refers to  a strong $\mathcal{C}$-hyperideal if for each $D \in \mathfrak{U}$ and  $D \cap B \neq \varnothing$ imply $D \subseteq B$ such that $\mathfrak{U}=\{\sum_{i=1}^n C_i \ \vert \ C_i \in \mathcal{C}, n \in \mathbb{N}\}$ and $\mathcal{C} = \{a_1 \circ a_2 \circ . . . \circ a_n \ \vert \ a_i \in A, n \in \mathbb{N}\}$ (for more details see \cite{phd}). 
\begin{definition} \cite{ameri} 
A proper hyperideal $B$ of 
 $A$ is maximal in $A$ if for
each hyperideal $M$ of $A$ with $B \subset M \subseteq A$, then $M = A$.
\end{definition}
 The intersection of all maximal hyperideals of $A$ is denoted by $J(A)$.
Also, $A$ refers to a local multiplicative hyperring if it has just one maximal hyperideal.   
\begin{definition} \cite{ameri}  Assume that  $B_1$ and $B_2$ are hyperideals of $H$. We define $(B_2:B_1)=\{x \in A \ \vert \ x \circ B_1 \subseteq B_2\}$.
%Assume that  $A,B$ are hyperideals of $H$ and $C \subseteq H$. We said that $A$ and $B$ are coprime if $A+B=H$.
\end{definition}
%%%%%%%%%%%%%%%%%%%%%%%%%%%%%%%%%%%%%%%%%%%%%%%%%%%%%%%%%%%%%%%%%%%
\section{  $(u,v)$-absorbing  hyperideals }
 A proper hyperideal $Q$ in $A$ refers to a $v$-absorbing hyperideal for $v \in \mathbb{N}$ if whenever $x_1 \circ \cdots \circ x_{v+1} \subseteq Q$ for $x_1,\cdots,x_{v+1} \in A$, then there exist $v$ of the $x_i^,$s whose product is a subset of $Q$. In this section, we aim to introduce the class of the  $(u,v)$-absorbing  hyperideals  as an extension of the $v$-absorbing hyperideals and present some of their properties. 
\begin{definition}
Let $u,v \in \mathbb{N}$ with $u >v$. A proper hyperideal $Q$ of $A$  is called  a   $(u,v)$-absorbing  hyperideal  if whenever    $x_1 \circ \cdots \circ x_u \subseteq Q$ for $x_1,\cdots, x_u \in A \backslash U(A)$, then there are $v$ of the $x_i^,$s whose product is a subset of  $Q$.
\end{definition}
\begin{example} \label{ex1}
Consider the multiplicative hyperring  $(\mathbb{Z}_T,+,\circ)$ where $T=\{2,4\}$. Then $Q=\langle 150 \rangle$ is a $(4,3)$-absorbing hyperideal of the multiplicative hyperring $(\mathbb{Z}_T,+,\circ)$. However, it is not $(3,2)$-absorbing as the fact that $3 \circ 5 ^2=\{300,600,1200\} \subseteq Q$ but $3 \circ 5=\{30,60\} \nsubseteq  Q$ and $5^2=\{50,100\} \nsubseteq  Q$.
\end{example}
\begin{remark} \label{t1}
Assume that $Q$ is a proper hyperideal of $A$ and $k,r,u,v \in  \mathbb{N}$. 
 \begin{itemize}
\item[\rm(i)]~ Let $Q$ be a $(u,v)$-absorbing hyperideal of $A$. Then: 
 \begin{itemize}
\item[\rm(1)]~ $Q$ is a $(k,r)$-absorbing hyperideal for every $k \geq u$ and $r \geq v$.
\item[\rm(2)]~ $Q$ is a $(u-1)$-absorbing hyperideal of $A$.
 \end{itemize}
\item[\rm(ii)]~ $Q$ is a $(u,v)$-absorbing hyperideal of $A$ if and only if $x_1 \circ \cdots \circ x_k \subseteq Q$ for $x_1,\cdots,x_k \in A \backslash U(A)$ and $k \geq u$ imply that there are $v$ of the $x_i^,$s whose product is a subset of $Q$.
 \end{itemize}
\end{remark}
\begin{proposition} \label{t2}
Let $Q_i$ be a proper hyperideal of $A$ and $u_i,v_i \in \mathbb{N}$ for every $i \in \{1,\cdots,k\}$. If $Q_i$ is a $(u_i,v_i)$-absorbing hyperideal of $A$ for every $i \in \{1,\cdots,k\}$, then $Q=\cap_{i=1}^k Q_i$ is a $(u,v)$-absorbing hyperideal of $A$ where $v=v_1+\cdots+v_k$ and $u=\max\{u_1,\cdots,u_k,v+1\}$.
\end{proposition}
\begin{proof}
Assume that $Q_i$ is a $(u_i,v_i)$-absorbing hyperideal of $A$ for every $i \in \{1,\cdots,k\}$. By Remark \ref{t1} (i), we conclude that $Q_i$ is a $(u,v_i)$-absorbing hyperideal of $A$ where  $v=v_1+\cdots+v_k$ and $u=\max\{u_1,\cdots,u_k,v+1\}$. Let $x_1 \circ \cdots \circ x_u \subseteq Q$ for $x_1,\cdots,x_u \in A \backslash U(A)$. Since $Q_i$ is a $(u,v_i)$-absorbing hyperideal of $A$ for every $i \in \{1,\cdots,k\}$ and  $x_1 \circ \cdots \circ x_u \subseteq Q_i$, we get the result that $\Pi_{\alpha \in I_i} x_{\alpha} \subseteq Q_i$ for each  $i \in \{1,\cdots,k\}$ such that $I_i \subseteq \{1,\cdots,u\}$ and $\vert I_i \vert =v_i$. Put $I= \cup_{i=1}^k I_i$. Since  $\Pi_{\alpha \in I} x_{\alpha} \subseteq Q$ and $\vert I \vert \leq v$, we conclude that $Q=\cap_{i=1}^k Q_i$ is a $(u,v)$-absorbing hyperideal of $A$.
\end{proof}
Let $u,v \in \mathbb{N}$. Then  the smallest integer not less than $\frac{u}{v} $ is denoted by $\lceil \frac{u}{v} \rceil$.
\begin{proposition}\label{t3}
Let  $Q$ be  a $(u,v)$-absorbing $\mathcal{C}$-hyperideal of $A$ for $u,v \in  \mathbb{N}$ with $u >v$. Then:
 \begin{itemize}
\item[\rm(i)]~ $rad(Q) =\{x \in A \ \vert \ x^v \subseteq Q \}$.
\item[\rm(ii)]~ $rad(Q)$ is a $(k,v)$-absorbing hyperideal of $A$ where $k=\max\{\lceil \frac{u}{v} \rceil, v+1\}$.
 \end{itemize}
\end{proposition}
\begin{proof}
(i) Assume that $Q$ is  a $(u,v)$-absorbing hyperideal of $A$. Take any $x \in rad(Q)$. Then there exists a smallest element $n \in \mathbb{N}$ such that $x^n \subseteq Q$. Let $n >v$. Remark \ref{t1}(i)(2) shows that  $Q$ is a $(u-1)$-absorbing hyperideal of $A$. Then we conclude that $n \leq u-1$ by Theorem 3.6 in \cite{anb4}. This means $x^u \subseteq Q$ which means $x^v \subseteq Q$ as $Q$ is  a $(u,v)$-absorbing hyperideal of $A$. This is contradicts the fact that $n$ is a smallest positive integer  such that $x^n \subseteq Q$. Therefore $n \leq v$ and $x^v \subseteq Q$.

(ii) Suppose that $x_1 \circ \cdots \circ x_k \subseteq rad(Q)$ such that  $k=\max\{\lceil \frac{u}{v} \rceil, v+1\}$ and $x_1,\cdots,x_k \in A \backslash U(A)$. Then for any $a \in x_1 \circ \cdots \circ x_k$, $a^v \subseteq Q$. Again, $a^v \subseteq (x_1 \circ \cdots \circ x_k)^v \subseteq x_1^v \circ \cdots \circ x_k^v \subseteq Q$. So, $(x_1^v \circ \cdots \circ x_k^v) \cap Q \neq \varnothing$ and then $x_1^v \circ \cdots \circ x_k^v \subseteq Q$ as $Q$ is a $\mathcal{C}$-hyperideal. The hyperideal $Q$ is a $(kv,v)$-absorbing hyperideal of $A$ because $k \geq \lceil \frac{u}{v} \rceil$. Then we conclude that $x_{i_1}^{r_1} \circ \cdots \circ x_{i_v}^{r_v} \subseteq Q$ for some $\{i_1,\cdots,i_v\} \subseteq \{1,\cdots,k\}$ such that $0 \leq i_1,\cdots,i_v \leq v$ and $r_1+\cdots+r_v=v$. This implies that $x_{i_1}^{v} \circ \cdots \circ x_{i_v}^{v} \subseteq Q$ which means $x_{i_1} \circ \cdots \circ x_{i_v} \subseteq rad(Q)$. Consequently,  $rad(Q)$ is a $(k,v)$-absorbing hyperideal of $A$.
\end{proof}
\begin{theorem}\label{t4}
Assume that $Q$ is a hyperideal of $A$ and $u,v \in  \mathbb{N}$ with $u >v$. Then there exists a  $(u,v)$-absorbing hyperideal of $A$ which is minimal in the set of all $(u,v)$-absorbing hyperideals of $A$ containing $Q$.
\end{theorem}
\begin{proof}
Suppose that $\Gamma=\{P \ \vert \ P \  \text{is a $(u,v)$-absorbing hyperideal of $A$ and $Q \subseteq P$} \}$. Let $M$ be a maximal hyperideal of $A$ containing  $Q$. Then $M$ is a $(u,v)$-absorbing hyperideal   for all $u,v \in  \mathbb{N}$ with $u >v$ and so $M \in \Gamma$ which means $\Gamma \neq \varnothing$. Order $\Gamma$ by $\leq$ such that $P_1 \leq P_2$ if and only if $P_1 \supseteq P_2$ for every $P_1, P_2 \in \Gamma$. Let $\{P_i\}_{i \in \Delta}$ be a non-empty chain of hyperideals in $\Gamma$. Put $I=\cap_{i \in \Delta} P_i$. Assume that $x_1 \circ \cdots \circ x_u \subseteq I$ for $x_1,\cdots,x_u \in A \backslash U(A)$ such that none of the product of the terms $v$ of $x_i^,$s is a subset of $I$. Then that there exists $P_j$ in the chain such that none of the product of the terms $v$ of $x_i^,$s is a subset of $P_j$ which is impossible. Hence there are $v$ of the $x_i^,$s whose product is a subset of  $I$. Now, by Zorn$^,$s lemma, we conclude that $(\Gamma, \leq)$ has a maximal element which means $\Gamma$ has a minimal  $(u,v)$-absorbing hyperideal of $A$ containing $Q$.
\end{proof}
As an immediate consequence of the previous theorem, we have the following
result.
\begin{corollary}\label{t5}
Suppose that $Q$ is a hyperideal of $A$ and $v \in \mathbb{N}$. Then there exists a minimal $v$-absorbing hyperideal $P$ of $A$ such that $Q \subseteq P$.  
\end{corollary}
\begin{theorem}\label{t6}
Assume that $Q$ is a $(u,v)$-absorbing hyperideal of $A$ and $k,r \in \mathbb{N}$ with $2 \leq k \leq r \leq u-1$. If $Q$ has $r$ minimal prime hyperideals, then there are at least $\frac{r!}{(r-k)! k!}$ $k$-absorbing hyperideals of $A$ that are minimal over $Q$. 
\end{theorem}
\begin{proof}
Suppose that $Q$ has $r$ minimal prime hyperideals and $k=2$. Then there exist the distinct prime hyperideals $Q_1,\cdots,Q_r$ of $A$ that are minimal over $Q$. Hence $Q_i \cap Q_j$ is a 2-absorbing hyperideal of $A$ for each $i,j \in \{1,\cdots,r\}$ and $i \neq j$ by Theorem 3.4 in \cite{anb4}. Hence we have $Q \subseteq Q_{ij} \subseteq Q_i \cap Q_j$ for some minmal 2-absorbing hyperideal $Q_{ij}$ of $A$ by Corollary \ref{t5}. Therefore we conclude that $Q_i$ and $Q_j$ are only minimal prime hyperideals over $Q_{ij}$ by Theorem 3.8 in \cite{anb4}. Assume that ${Q_{ij}}^,$s are not distinct such that $Q_{ij}=Q_{st}$ for some $i,j,s,t \in \{1,\cdots,r\}$ and $\{i,j\} \neq \{s,t\}$. Then we conclude that $Q_i$, $Q_j$, $Q_s$ and $Q_t$ are distinct minimal prime hyperideals over $Q_{ij}$ which contradicts the fact that there are are at most two prime hyperideals of $A$ that are minimal over the 2-absorbing hyperideal $Q_{ij}$. Hence the ${Q_{ij}}^,$s are distinct. Thus we get the result that $Q$ has at least $\frac{r!}{(r-k)! k!}$ minimal $2$-absorbing hyperideals. If $k>2$, one can easily complete the proof by using a similar argument.
\end{proof}
\begin{lem} \label{t7}
Let $Q$ be a $(u,v)$-absorbing hyperideals of $A$, $Q_1,\cdots,Q_v$ be incomparable prime $\mathcal{C}$-hyperideals of $A$ such that $Q \subseteq \cap_{i=1}^v Q_i$ and $k_1,\cdots,k_v \in \mathbb{N}$. If $x_1^{k_1} \circ \cdots \circ x_v^{k_v} \subseteq Q$ for $x_i \in Q_i \backslash (\cup_{j \neq i}Q_j)$, then  $x_1 \circ \cdots \circ x_v \subseteq Q$.
\end{lem}
\begin{proof}
Suppose that $k_1,\cdots,k_v \in \mathbb{N}$. Put $k=k_1+\cdots+k_v$.  Let $k \geq u$. So $Q$ is a $(k,v)$-absorbing hyperideal of $A$ by Remark \ref{t1}(i). Then there exist integers $r_1,\cdots,r_v$ such that $x_1^{r_1} \circ \cdots \circ x_v^{r_v} \subseteq Q$, $0 \leq r_i \leq k_i$ and $r_1+\cdots+r_v=v$. Assume that one of the $r_i^,$s is zero. Then we have $x_1^{r_1} \circ \cdots x_{i-1}^{r_{i-1}} \circ x_{i+1}^{r_{i+1}}\circ \cdots  \circ x_v^{r_v} \subseteq Q \subseteq Q_i$. Then there exists $j \in \{1,\cdots,i-1,i+1,\cdots, v\}$ such that $x_j \in Q_i$ which is impossible. Hence $x_1 \circ \cdots \circ x_v \subseteq Q$.
\end{proof}
\begin{theorem} \label{t8}
Assume that $Q$ is a $(u,v)$-absorbing strong $\mathcal{C}$-hyperideal of $A$ for $u,v \in \mathbb{N}$ with $2 \leq v < u$ and $Q_1,\cdots,Q_v$ are the only minimal prime $\mathcal{C}$-hyperideals over $Q$. If $z_i \in Q_i \backslash (\cup_{j \neq i}Q_j)$, then $Q_j \circ \Pi_{i \neq j} z_i \subseteq Q$.
\end{theorem}
\begin{proof}
Assume that $x \in P_j$. If $x \notin \cup_{j \neq i}Q_j$, then we have $x \circ \Pi_{i \neq j} z_i \subseteq \cap_{i=1}^vQ_i=rad(Q)$ as the $Q_i^,$s are the only minimal prime hyperideals over $Q$. Therefore we conclude that $x \circ \Pi_{i \neq j} z_i \subseteq Q$ by Proposition \ref{t3} and Lemma \ref{t7}. Let $x \in Q_j \
\cap (\cup_{j \neq i}Q_j)$. Then there exist $y \in A$ and $z \in Q_j \backslash (\cup_{j \neq i}Q_j)$ such that $w+x \in Q_j \backslash (\cup_{j \neq i}Q_j)$ for every  $w \in y \circ z$ by Corollary 3.12 in \cite{anb4}. It follows that   $(w+x) \circ \Pi_{i \neq j} z_i \subseteq Q$ and so $w \circ \Pi_{i \neq j} z_i + x \circ \Pi_{i \neq j} z_i \subseteq Q$ as $Q$ is a strong $\mathcal{C}$-hyperideal of $A$. Since $w \circ \Pi_{i \neq j} z_i  \subseteq Q$, we obtain  $x \circ \Pi_{i \neq j} z_i \subseteq Q$ which means $Q_j \circ \Pi_{i \neq j} z_i \subseteq Q$.
\end{proof}

 Let $Q$ be a $u$-absorbing hyperideal of $A$ for $u \in \mathbb{N}$.  We define $Abs(Q)=\min \{v \ \vert \ Q \ \text{is a $v$-absorbing hyperideal of } A \}$, otherwise $Abs(Q) = \infty$ \cite{anb4}. Moreover, we define $abs(Q)=\min \{v \ \vert \ Q \ \text{is an $(Abs(Q)+1,v)$-absorbing hyperideal of } A \}$. It is clear that  $abs(Q) \leq Abs(Q)$. Thus, in Proposition \ref{t2}, we get $abs(\cap_{i=1}^k) \leq abs(Q_1)+\cdots+abs(Q_k)$ and $Abs(\cap_{i=1}^k Q_i) \leq \max\{Abs(Q_1),\cdots,Abs(Q_k),abs(Q_1)+\cdots+abs(Q_k)\}$. Also, in Proposition \ref{t3}, we have $abs(rad(Q)) \leq abs(Q)$ and $Abs(rad(Q)) \leq \max\{\lceil \frac{Abs(Q)+1}{abs(Q)} \rceil-1,abs(Q)\}.$ Recall from \cite{ameri} that the hyperideals $I$ and $J$ are coprime if $I+J=A$.
 \begin{theorem} \label{t9}
Assume that $Q$ is a $(u,v)$-absorbing strong $\mathcal{C}$-hyperideal of $A$ for $u,v \in \mathbb{N}$ with $2 \leq v < u$ and $Q_1,\cdots,Q_v$ are the only minimal prime $\mathcal{C}$-hyperideals over $Q$. Then 
\begin{itemize}
\item[\rm{(i)}]~ $Q_1 \circ \cdots \circ Q_v \subseteq Q$ and $abs(Q)=v$. 
\item[\rm{(ii)}]~ If $Q_1,\cdots,Q_v$ are pairwise coprime, then $Q_1 \circ \cdots \circ Q_v =Q$ and $abs(Q)=v$. 
 \end{itemize}
 \end{theorem} 
 \begin{proof}
 It can be easily seen that the claim is true in a similar manner to the proofs of
Theorem 3.13 and Theorem 3.14 in \cite{anb4}.
 \end{proof}
 \begin{theorem} \label{t10}
 Assume that $Q$ is a radical $\mathcal{C}$-hyperideal of $A$. Then $Q$ is a $(u,v)$-absorbing hyperideal of $A$ if and only if $Q$ is a $v$-absorbing hyperideal of $A$. 
 \end{theorem}
 \begin{proof}
 ($\Longrightarrow$) Suppose that $Q$ is a $(u,v)$-absorbing hyperideal of $A$ and $x_1 \circ \cdots \circ x_{u-1} \subseteq Q$ for $x_1,\cdots,x_{u-1} \in A \backslash U(A)$. Since $x_1 \circ \cdots \circ x_{k-1} \circ x_k^2 \circ x_{k+1} \circ \cdots \circ x_{u-1} \subseteq Q$, we conclude that either $x_{i_1} \circ \cdots \circ x_{i_v} \subseteq Q$ for some $ i_j \in \{1,\cdots,k-1,k,k+1,\cdots,u-1\}$ which implies $Q$ is a $(u-1,v)$-absorbing hyperideal of $A$ or $x_k^2 \circ x_{i_1} \circ   \cdots \circ x_{i_{v-2}} \subseteq Q$ for some $ i_j \in \{1,\cdots,k-1,k+1,\cdots,u-1\}$ which means $x_k \circ x_{i_1} \circ   \cdots \circ x_{i_{v-2}} \subseteq rad(Q)$. Since $Q$ is a radical hyperideal, we get $x_k \circ x_{i_1} \circ   \cdots \circ x_{i_{v-2}} \subseteq Q$ which implies  $Q$ is a $(u-1,v)$-absorbing hyperideal of $A$.  By continuing this process, we conclude $Q$ is a $(v+1,v)$-absorbing hyperideal and so it is a $v$-absorbing hyperideal of $A$ by Remark \ref{t1}(i).
 
 ($\Longleftarrow$) Let $Q$ is a $v$-absorbing hyperideal of $A$ and $x_1 \circ \cdots \circ  x_u \subseteq Q$ for $x_1,\cdots,x_{u-1} \in A \backslash U(A)$. Take any $c \in (x_{i_1} \circ \cdots \circ x_{i_{u-v}}) \backslash Q$ such that $i_j \in \{1,\cdots, u\}$. Then a product of $v$ of the $x_i^,$s othere than the $x_{i_j}$ with $c$ is a subset of $Q$ and so a product of $v$ of the $x_i^,$s othere than the $x_{i_j}$ is a subset of $Q$ as  $Q$ is a $v$-absorbing hyperideal of $A$. Thus $Q$ is a $(u,v)$-absorbing hyperideal of $A$.
 \end{proof}
 Let us give the following theorem without proof as a result of Theorem \ref{t10}.
 \begin{corollary}
Let  $Q$ be   a $(u,v)$-absorbing radical $\mathcal{C}$-hyperideal of $A$. Then $abs (Q)=Abs(Q)$.
 \end{corollary}
 Recall from \cite{das} that a proper hyperideal $Q$ of $A$ is primary if for any $x,y \in A$, $x \circ y \subseteq Q$ implies that $x \in Q$ or $y^t \subseteq Q$ for some $t \in \mathbb{N}$. Since the radical of a primary $\mathcal{C}$-hyperideal $Q$  is a prime hyperideal of $A$, say $P$,  $Q$ is referred to as a $P$-primary $\mathcal{C}$-hyperideal of $A$.
 \begin{theorem}
 Assume that $u,v \in \mathbb{N}$ and $Q$ is a $P$-primary $\mathcal{C}$-hyperideal of $A$ where $P$ is a prime hyperideal of $A$ with $P^v \subseteq Q$. Then $Q$ is a $(u,v)$-absorbing hyperideal of $A$ for any $u \geq v+1$. Furthermore, if $P^v$ is a $P$-primary $\mathcal{C}$-hyperideal of $A$, then $P^v$ is a $(u,v)$-absorbing hyperideal of $A$ for any $u \geq v+1$ and $abs(P^v)=v$.
 \end{theorem}
 \begin{proof}
 The proof follows from Theorem 4.1 in \cite{anb4}.
 \end{proof}
\section{$AB$-$(u,v)$-absorbing hyperideals}
This section is devoted to the study of  a subclass of the $(u,v)$-absorbing hyperideals called $AB$-$(u,v)$-absorbing hyperideals.
\begin{definition}
Let $u,v \in \mathbb{N}$ with $u >v$. A proper hyperideal $Q$ of $A$  is said to be  an   $AB$-$(u,v)$-absorbing  hyperideal  if whenever    $x_1 \circ \cdots \circ x_u \subseteq Q$ for $x_1,\cdots, x_u \in A$, then there are $v$ of the $x_i^,$s whose product is a subset of  $Q$.
\end{definition}
\begin{theorem}
Let $Q$ be a proper strong $\mathcal{C}$-hyperideal of $A$ such that $a+1 \notin U(A)$ for some $a \in Q$ and $u,v \in \mathbb{N}$ with $u >v$. Then $Q$ is an $AB$-$(u,v)$-absorbing  hyperideal if and only if $Q$ is a $(u,v)$-absorbing  hyperideal.
\end{theorem}
\begin{proof}
$(\Longrightarrow)$ It is obvious.

$(\Longleftarrow)$ Assume that $x_1 \circ \cdots \circ x_u \subseteq Q$ for $x_1,\cdots, x_u \in A$. Let $x_1,\cdots, x_{s-1}$ be non-unit elememts in $A$ and  $x_s,  \cdots, x_u$ be unit elements in $A$ for some $s \geq v+2$. Since $x_1 \circ \cdots \circ x_u \subseteq Q$ and $x_s,  \cdots, x_u$ are units in $A$, we get $x_1 \circ \cdots,\circ x_s \subseteq Q$. Put $y_i=1+a$ for each $i \in \{1,\cdots,s\}$. Then the $y_i^,$s are non-unit. Since $Q$ is a $(u,v)$-absorbing  hyperideal and $x_1 \circ \cdots \circ x_{s-1} \circ y_s \circ \cdots \circ y_u \subseteq Q$, we conclude that there exist $v$ elements of the $x_i^,$ and the $y_j^,$s whose product is a subset of $Q$. Take any $n \in \mathbb{N}$. Assume that $(a+1)^n \subseteq Q$. This implies that $(\Sigma_{k=0}^n \dbinom{n}{k} a^{n-k} \circ 1^k) \cap Q \neq \varnothing $. Since $Q$ is a  strong $\mathcal{C}$-hyperideal of $A$, we have $(\Sigma_{k=0}^n \dbinom{n}{k} a^{n-k} \circ 1^k) \subseteq Q$ and so $1 \in Q$, a contradiction. Then $(a+1)^n \nsubseteq Q$ for all $n \geq 1$. Thus $x_{i_1} \circ \cdots \circ x_{i_h} \circ (a+1)^{v-h} \subseteq Q$ for some $h \in \{1,\cdots,v\}$ and $i_j \in \{1,\cdots,s-1\}$. Clearly, $\Sigma_{k=0}^{v-h} \dbinom{v-h}{k} a^{v-h-k} \circ 1^k=X+1$ where $X \subseteq Q$. Since $x_{i_1} \circ \cdots \circ x_{i_h} \circ (X+1) \subseteq Q \cap  (x_{i_1} \circ \cdots \circ x_{i_h} \circ X)+(x_{i_1} \circ \cdots \circ x_{i_h} \circ 1)$, we obtain  $(x_{i_1} \circ \cdots \circ x_{i_h}\circ X)+(x_{i_1} \circ \cdots \circ x_{i_h} \circ 1) \subseteq Q$. Hence we get the result that $x_{i_1} \circ \cdots \circ x_{i_h}  \subseteq Q$ which implies $Q$ is an $AB$-$(u,v)$-absorbing  hyperideal of $A$.
\end{proof}
As a result of the previous  theorems, we give the following result.
\begin{corollary}
Let $Q$ be a proper strong $\mathcal{C}$-hyperideal of $A$ such that $Q \nsubseteq J(A)$ and $u,v \in \mathbb{N}$ with $u >v$. Then $Q$ is an $AB$-$(u,v)$-absorbing  hyperideal if and only if $Q$ is a $(u,v)$-absorbing  hyperideal.
\end{corollary}
Let $X$ be  a finite sum of finite products of elements of $A$. Consider the relation $\gamma$ on a multiplicative hyperring $A$ defined as $x \gamma y$ if and only if $\{x,y\} \subseteq X$, namely,  
$x \gamma y $ if and only if    $\{x,y\} \subseteq \sum_{j \in J} \prod_{i \in I_j} z_i$ for some $ z_1, ... , z_n \in A$ and $ I_j, J \subseteq \{1,... , n\}$.  $\gamma^{\ast}$ denotes the transitive closure of $\gamma$. The relation $\gamma^{\ast}$ is the smallest equivalence relation on $A$ such that the set of all equivalence classes, i.e., the
quotient $A/\gamma^{\ast}$,  is a fundamental ring. Assume that $\Sigma$
is the set of all finite sums of products of elements of $A$.  We can rewrite the
definition of $\gamma ^{\ast}$  on $A$, namely, 
$x\gamma^{\ast}y$ if and only if  there exist  $ z_1, ... , z_n \in A$ such that $z_1 = x, z_{n+1 }= y$ and $u_1, ... , u_n \in \Sigma$ where
$\{z_i, z_{i+1}\} \subseteq u_i$ for $1 \leq i \leq n$.
Suppose that $\gamma^{\ast}(x)$ is the equivalence class containing $x \in A$. Define $\gamma ^{\ast}(x) \oplus \gamma^{\ast}(y)=\gamma ^{\ast}(z)$ for every  $z \in \gamma^{\ast}(x) + \gamma ^{\ast}(y)$ and $\gamma ^{\ast}(x) \odot \gamma ^{\ast}(y)=\gamma ^{\ast}(w)$ for every  $w \in  \gamma^{\ast}(x) \circ \gamma ^{\ast}(y)$. Then $(A/\gamma ^{\ast},\oplus,\odot)$ is a ring called a fundamental ring of $A$ \cite{sorc4}.

\begin{theorem}\label{7}
Let $Q$ be a proper hyperideal of $A$. Then  $Q$  is an $AB$-$(u,v)$-absorbing hyperideal of $(A,+,\circ)$ if and only if $Q/\gamma ^{\ast}$ is  an  $AB$-$(u,v)$-absorbing hyperideal  of $(A/\gamma ^{\ast},\oplus,\odot)$. 
\end{theorem}
\begin{proof}
($\Longrightarrow$) Let $x_1 \odot \cdots \odot x_u \in Q/\gamma ^{\ast}$  for some $x_1,\cdots,x_u \in A/\gamma ^{\ast}$. Then there exist $a_i \in A$   for every $i \in \{1,\cdots,u\}$ such that $x_i =\gamma^{\ast}(a_i)$. This implies  that $x_1\odot \cdots \odot x_u=\gamma^{\ast}(a_1) \odot \cdots \odot \gamma^{\ast}(a_u)=\gamma^{\ast}(a_1 \circ \cdots  \circ a_u)$. Since $
\gamma^{\ast}(a_1 \circ \cdots  \circ a_u) \in Q/\gamma^{\ast}$, we get $ a_1 \circ \cdots  \circ a_u \subseteq  Q$. Since $Q$ is an $AB$-$(u,v)$-absorbing hyperideal of  $A$, we get the result that  there is $v$ elements of the $a_i^,$s, say $a_{i_1},\cdots,a_{i_v}$,  whose product is a subset of $Q$. This implies that  $x_{i_1} \odot \cdots \odot x_{i_v}= \gamma^{\ast}(a_{i_1}) \odot \cdots \odot \gamma^{\ast}(a_{i_v})=\gamma^{\ast}(a_{i_1} \circ \cdots \circ a_{i_v})\in Q/\gamma ^{\ast}$. Consequently,   $Q/\gamma ^{\ast}$ is an $AB$-$(u,v)$-absorbing hyperideal  of $A/\gamma ^{\ast}$.

($\Longleftarrow$) Suppose that $a_1 \circ \cdots  \circ a_u \subseteq Q$ for some $a_1, \cdots, a_u \in A$. Then we have  $\gamma^{\ast}(a_1), \cdots, \gamma^{\ast}(a_u) \in A/\gamma^{\ast} $ and so
$\gamma^{\ast}(a_1) \odot \cdots  \odot\gamma^{\ast}(a_u)= \gamma^{\ast}(a_1 \circ \cdots  \circ a_u) \in Q/\gamma^{\ast}$. Since $Q/\gamma ^{\ast}$ is an $AB$-$(u,v)$-absorbing hyperideal  of  $A/\gamma ^{\ast}$,  there exist $v$ elements of the  $\gamma^{\ast}(a_i)^,$s, say $\gamma^{\ast}(a_{i_1}), \cdots,\gamma^{\ast}(a_{i_v})$ whose product is in  $ Q/\gamma^{\ast}$. This means    $\gamma^{\ast}(a_{i_1} \circ \cdots \circ a_{i_v})=\gamma^{\ast}(a_{i_1}) \odot \circ \odot  \gamma^{\ast}(a_{i_v}) \in Q/\gamma^{\ast}$ which implies  $a_{i_1} \circ \cdots \circ a_{i_v} \subseteq Q$. Thus $Q$ is an $AB$-$(u,v)$-absorbing hyperideal  of  $A$.
\end{proof}

Let $(A_1,+_1,\circ_1)$ and $(A_2,+_2,\circ_2)$ be two multiplicative hyperrings with nonzero identity.  The set $A_1 \times A_2$  with the operation $+$ and the hyperoperation $\circ$  defined  as

$(x_1,x_2)+(y_1,y_2)=(x_1+_1y_1,x_2+_2y_2)$

$(x_1,x_2) \circ (y_1,y_2)=\{(x,y) \in A_1 \times A_2 \ \vert \ x \in x_1 \circ_1 y_1, y \in x_2 \circ_2 y_2\}$ \\
is a multiplicative hyperring \cite{ul}. Now, we present some characterizations of $AB$-$(u,v)$-absorbing hyperideals on cartesian product of commutative multiplicative hyperring.

\begin{theorem}
Assume that  $A=A_1  \times A_2$ such that $A_1, A_2$ are commutative multiplicative hyperrings,  $Q=Q_1 \times A_2$ such that $Q_1$ is   a proper  hyperideal of $A_1$ and $u,v \in \mathbb{N}$ with $u >v$. Then the following are equivalent.
 \begin{itemize}
\item[\rm(i)]~ $Q$ is a  $(u,v)$-absorbing  hyperideal of $A$.
\item[\rm(ii)]~ $Q_1$ is an $AB$-$(u,v)$-absorbing  hyperideal of $A_1$.
\item[\rm(iii)]~ $Q$ is an $AB$-$(u,v)$-absorbing  hyperideal of $A$.
\end{itemize}
\end{theorem}
\begin{proof}
(i) $\Longrightarrow$ (ii) Suppose on the contrary, $Q_1$ is not an $AB$-$(u,v)$-absorbing  hyperideal of $A_1$. This means that we have $x_1 \circ_1 \cdots \circ_1 x_u \subseteq Q_1$ for some $x_1, \cdots,x_u \in A_1$ and none of product of the terms $v$ of the $x_i^,$s is not a subset of $Q_1$. Since $Q$ is a  $(u,v)$-absorbing  hyperideal of $A$, $(x_1,0) \circ \cdots \circ (x_u,0) \subseteq Q$ and none of the $(x_i,0)^,$s are unit, there exist $v$ of the $(x_i,0)^,$s whose is a subset of $Q$ which implies  there exist $v$ of the $x_i^,$s whose is a subset of $Q_1$ which is a contradiction. Hence $Q_1$ is an $AB$-$(u,v)$-absorbing  hyperideal of $A_1$.

(ii) $\Longrightarrow$ (iii) Let $(x_1,y_1) \circ \cdots \circ (x_u,y_u) \subseteq Q$ for $(x_1,y_1), \cdots ,  (x_u,y_u) \in A$. This means that $x_1 \circ_1 \cdots \circ_1 x_u \subseteq Q_1$. Since $Q_1$ is an $AB$-$(u,v)$-absorbing  hyperideal of $A_1$, there exist $v$ of the $x_i^,$s, say $x_{i_1},\cdots,x_{i_v}$, whose product is a subset of $Q_1$. It follows that $(x_{i_1},y_{i_1}) \circ \cdots \circ (x_{i_v},y_{i_v}) \subseteq Q$ and so $Q$ is an $AB$-$(u,v)$-absorbing  hyperideal of $A$.

(iii) $\Longrightarrow$ (i) The claim is true because every $AB$-$(u,v)$-absorbing is a $(u,v)$-absorbing  hyperideal.
\end{proof}
\section{  $(u,v)$-absorbing prime hyperideals }
In this section, we aim to extend the class of the $(u,v)$-absorbing hyperideals to $(u,v)$-absorbing prime hyperideals.
\begin{definition}
Let $u,v \in \mathbb{N}$ with $u >v$. A proper hyperideal $Q$ of  $A$  refers to a   $(u,v)$-absorbing prime hyperideal if $x_1 \circ \cdots \circ x_u \subseteq Q$ for $x_1,\cdots, x_u \in A \backslash U(A)$ implies either $x_1 \circ \cdots \circ x_v \subseteq Q$ or $x_{v+1} \circ \cdots \circ x_u \subseteq Q$.
\end{definition}
\begin{example}
In Example \ref{ex1}, the hyperideal $\langle 15 \rangle$ is a $(3,1)$-absorbing prime  hyperideal. However, since  $3^2 \circ 5=\{180,360,720 \} \subseteq \langle 15 \rangle$ but $3,5 \notin \langle 15 \rangle$, then $\langle 15 \rangle$ is not a $(3,1)$-absorbing hyperideal
\end{example}
\begin{remark} \label{s1}
Let $u,v \in \mathbb{N}$ with $u >v$ and $Q$ be a $\mathcal{C}$-hyperideal of $A$. 
\begin{itemize}
\item[\rm{(i)}]~  $Q$ is a $(u,v)$-absorbing prime hyperideal of $A$ if and only if  $Q$ is a $(u,u-v)$-absorbing prime hyperideal.
 \item[\rm{(ii)}]~ If $Q$ is a $(u,v)$-absorbing prime hyperideal of $A$, then   $Q$ is a $(u+1,v+1)$-absorbing prime hyperideal.
\end{itemize}
\end{remark}
\begin{proposition} \label{s2}
Assume that  $Q$ is a strong $\mathcal{C}$-hyperideal of $A$ such that $A$ is not local and $u,v \in \mathbb{N}$ with $u >v$. Then $Q$ is a $(u,v)$-absorbing prime hyperideal of $A$ if and only if $Q$ is a $(u-1,v-1)$-absorbing prime hyperideal.
\end{proposition}
\begin{proof}
$(\Longrightarrow)$ Let $Q$ be a $(u,v)$-absorbing prime hyperideal of $A$. Suppose on contrary that $Q$ is not a $(u-1,v-1)$-absorbing prime hyperideal of $A$. Then there exist $x_1,\cdots,x_{u-1} \in A \backslash U(A)$ such that $x_1\circ \cdots \circ x_{u-1} \subseteq Q$ but neither $x_1\circ \cdots \circ x_{v-1} \subseteq Q$ nor $x_{v}\circ \cdots \circ x_{u-1} \subseteq Q$. Take any $y \in A \backslash U(A)$. Since $Q$ is a $(u,v)$-absorbing prime hyperideal of $A$, $y \circ x_1\circ \cdots \circ x_{u-1} \subseteq Q$ and $x_{v}\circ \cdots \circ x_{u-1} \nsubseteq Q$, we have $y \circ x_1\circ \cdots \circ x_{v-1} \subseteq Q$. Take any  $z \in U(A)$. If $y+z \notin U(A)$, then we obtain $(y+z) \circ x_1\circ \cdots \circ x_{v-1} \subseteq Q$ as $(y+z) \circ x_1\circ \cdots \circ x_{u-1} \subseteq Q$, $x_{v}\circ \cdots \circ x_{u-1} \nsubseteq Q$ and $Q$ is a $(u,v)$-absorbing prime hyperideal of $A$. Since   $Q$ is a strong $\mathcal{C}$-hyperideal of $A$ and $(y+z) \circ x_1\circ \cdots \circ x_{v-1}   \subseteq y \circ x_1\circ \cdots \circ x_{v-1} +z  \circ x_1\circ \cdots \circ x_{v-1}$, we conclude that $y \circ x_1\circ \cdots \circ x_{v-1} +z  \circ x_1\circ \cdots \circ x_{v-1} \subseteq Q$. This implies that $ x_1\circ \cdots \circ x_{v-1} \subseteq Q$ because $y \circ x_1\circ \cdots \circ x_{v-1} \subseteq Q$. This is a contradiction. Then we have $y+z \in U(A)$. By Lemma 2.6 in \cite{Ghiasvand}, we get the result that $A$ is local which is  a contradiction. Thus $Q$ is a $(u-1,v-1)$-absorbing prime hyperideal of $A$.

$(\Longleftarrow) $ It is obvious by Remark \ref{s1} (ii). 
\end{proof}

\begin{proposition} \label{s3}
Assume that  $u,v \in \mathbb{N}$ with $u >v$ and $Q$ is a strong $\mathcal{C}$-hyperideal of $A$ such that $A$ is not local. Then $Q$ is a $(u-v+1,1)$-absorbing prime hyperideal of $A$ if and only if $Q$ is a prime hyperideal. 
\end{proposition}
\begin{proof}
($\Longrightarrow$)Let $Q$ be a $(u-v+1,1)$-absorbing prime hyperideal of $A$. Then we get the result that $Q$ is a $(u-v+1,u-v)$-absorbing prime hyperideal of $A$ by Remark \ref{s1} (i). Hence we conclude that $Q$ is a prime hyperideal by Proposition \ref{s2}. 

($\Longleftarrow$) Straightforward. 
\end{proof}
For  a  $\mathcal{C}$-hyperideal $Q$   of $A$  that the hyperring $A$ is not  local, it can be easily proved, by induction, that   $Q$ is a $(u-1,v-1)$-absorbing prime hyperideal for $u,v \in \mathbb{N}$ with $u >v$ if and only if  $Q$ is a $(u-v+1,1)$-absorbing prime hyperideal. Thus in view of Proposition \ref{s2} and Proposition \ref{s3}, we have the following result.
\begin{corollary} \label{s4}
In a hyperring  $A$ that is not  local, every $(u,v)$-absorbing prime hyperideal is prime and vice versa.   
\end{corollary}
\begin{theorem}
Assume that $Q_1$ and $Q_2$ are strong $\mathcal{C}$-hyperideals of the hyperring $A_1$ and $A_2$, respectively. Let  $u,v \in \mathbb{N}$ with $u >v$. Then the following are equivalent:
\begin{itemize}
\item[\rm{(i)}]~ $Q_1 \times Q_2$ is a $(u,v)$-absorbing prime hyperideal of $A_1 \times A_2$.
\item[\rm{(ii)}]~$Q_1 \times Q_2$ is a prime hyperideal of $A_1 \times A_2$.
\item[\rm{(iii)}]~ $Q_1$ is a prime hyperideal of $A_1$ and $Q_2=A_2$ or $Q_2$ is a prime hyperideal of $A_2$ and $Q_1=A_1$.
\end{itemize}
\end{theorem}
\begin{proof}
It can be easily seen that the idea is true by Corollary \ref{s4} as $A_1 \times A_2$ is not local.
\end{proof}
\begin{proposition} \label{s5}
Let $u,v \in \mathbb{N}$ with $u >v$ and  $Q$ be a strong $\mathcal{C}$-hyperideal of  $A$. Then $Q$ is a $(v+1,v)$-absorbing prime hyperideal of $A$ if and only if  $Q$ is a prime hyperideal or $A$ is local such that $M^v \subseteq Q$ for the only maximal hyperideal $M$ of $A$.
\end{proposition}
\begin{proof}
($\Longrightarrow$)  Let $Q$ be a $(v+1,v)$-absorbing prime hyperideal of $A$ and  the hyperring $A$ is not  local. Then $Q$ is prime by Corollary \ref{s4}. Suppose that $A$ is local and $M$ is the only maximal hyperideal of $A$ such that $M^v \nsubseteq Q$. If $Q$ is not a $(v,v-1)$-absorbing prime hyperideal of $A$, then there exist $x_1,\cdots,x_v \in A \backslash U(A)$ such that $x_1 \circ \cdots \circ x_v \subseteq Q$ but $x_1 \circ \cdots \circ x_{v-1} \nsubseteq Q$ and $x_v \notin Q$. Take any  $a \in a_1 \circ \cdots \circ a_v $ for every $a_1,\cdots,a_v \in M$. Since $Q$ is a $(v+1,v)$-absorbing prime hyperideal of $A$, 
$a \circ x_1 \circ \cdots \circ x_v \subseteq Q$ and $x_v \notin Q$, we get $a \circ x_1 \circ \cdots \circ x_{v-1} \subseteq Q$. Let $b \in x_1 \circ \cdots \circ x_{v-1}$. So $a \circ b \subseteq a \circ x_1 \circ \cdots \circ x_{v-1} \subseteq Q$. On the other hand, $a \circ b \subseteq a_1 \circ \cdots \circ a_v \circ b$. Then we get $a_1 \circ \cdots \circ a_v \circ b \subseteq Q$ as $Q$ is a strong $\mathcal{C}$-hyperideal of  $A$. If $b \in Q$, then $x_1 \circ \cdots \circ x_{v-1} \cap Q \neq \varnothing$ and so $x_1 \circ \cdots \circ x_{v-1} \subseteq Q$ which is impossible. Therefore $b \notin Q$. Since $Q$ is a $(v+1,v)$-absorbing prime hyperideal of $A$, $a_1 \circ \cdots \circ a_v \circ b \subseteq Q$ and $b \notin Q$, we conclude that $a_1 \circ \cdots \circ a_v \subseteq Q$ which means $M^v \subseteq Q$, a contradiction. Then we get the result that  $Q$ is  a $(v,v-1)$-absorbing prime hyperideal of $A$. By induction, we conclude that $Q$ is prime as $M^i \nsubseteq Q$ for any $i \leq v$.

($\Longleftarrow$) It is obvious.
\end{proof}
\begin{corollary} \label{s6}
Assume that $u,v \in \mathbb{N}$ with $u >v$ and  $Q$ is a $(v+1,v)$-absorbing prime strong $\mathcal{C}$-hyperideal of  $A$ that is not prime. If $P$ is a hyperideal of $A$ containing $Q$, then $P$ is a $(u,v)$-absorbing prime hyperideal of $A$.
\end{corollary}
\begin{proof}
Suppose that $Q$ is a $(v+1,v)$-absorbing prime strong $\mathcal{C}$-hyperideal of  $A$ that is not prime. By Proposition \ref{s5}, we get the result that $A$ is local and $M^v \subseteq Q$ where $M$ is the only maximal hyperideal of $A$.  Then $M^v \subseteq P$ and so we are done by By Proposition \ref{s5}.
\end{proof}
\begin{theorem} \label{s7}
Let $Q$ be a $(u,v)$-absorbing prime strong $\mathcal{C}$-hyperideal of $A$ where $u,v \in \mathbb{N}$ and $u >v$. Then the following statements hold:
\begin{itemize}
\item[\rm{(i)}]~ $rad(Q)$ is a prime hyperideal of $A$.
\item[\rm{(ii)}]~ If $Q$ is not a prime hyperideal of $A$, then $A$ is local. In particular, if $u=v+1$, then $rad(Q)=M$ where $M$ is the only maximal hyperideal of $A$.
\end{itemize}
\end{theorem}
\begin{proof}
(i) Suppose that $Q$ is a $(u,v)$-absorbing prime strong $\mathcal{C}$-hyperideal of $A$. Let $a \circ b \subseteq rad(Q)$ for $a,b \in A$. We may assume $a,b \in A \backslash U(A)$, without loss of generality. Then we have $a^n \circ b^n \subseteq Q$ for some $n \in \mathbb{N}$. Take any $x \in a^n$ and $y \in b^n$. If $x$ or $y \in U(A)$, then we obtain $b^n \subseteq e \circ b^n \subseteq   x^{-1} \circ x \circ b^n \subseteq x^{-1}  \circ a^n \circ b^n \subseteq Q$ which means $b \in rad(Q)$ or $a^n \subseteq a^n  \circ e \subseteq    a^n \circ y \circ y^{-1} \subseteq  a^n \circ b^n \circ y^{-1} \subseteq Q$ which means $a \in rad(Q)$. Now, we assume that $x,y \notin U(A)$.  Since $a^{v-1} \circ x \circ b^{u-v-1} \circ y \subseteq Q$ and $Q$ is a $(u,v)$-absorbing prime hyperideal of $A$, we get the result that  $a^{v-1} \circ x \subseteq Q$ or $b^{u-v-1} \circ y \subseteq Q$. In the first case, since $(a^{v-1} \circ a^n) \cap Q \neq \varnothing$, we get $a^{v+n-1}=a^{v-1} \circ a^n \subseteq Q$ which means $a \in rad(Q)$. In the second case, since $(b^{u-v-1} \circ b^n) \cap Q \neq \varnothing$, we have $b^{u-v+n-1}=b^{u-v-1} \circ b^n  \subseteq Q$ which implies $b \in rad(Q)$. Thus $rad(Q)$ is a prime hyperideal of $A$. 

(ii) Suppose that $Q$ is a $(u,v)$-absorbing prime strong $\mathcal{C}$-hyperideal of $A$ that is not prime. Then we conclude that $A$ is local by Corollary \ref{s4}. Let $u=v+1$ and $M$ be the only maximal hyperideal of $A$. By  Proposition \ref{s5}, we conclude that $M^v \subseteq Q$ and so $rad(Q)=M$.
\end{proof}

\begin{lem} \label{s10}
Let every  hyperideal of $A$ be a prime  $\mathcal{C}$-hyperideal. Then $A$ is a hyperfield. 
\end{lem}
\begin{proof}
%Let $0 \in x \circ y $ for $x,y \in A$. Since $0 \in x \circ y \cap \langle 0 \rangle$, we have $x \circ y \subseteq \langle 0 \rangle$. This implies that $x=0$ or $y=0$ as $\langle 0 \rangle$ is prime. Hence $A$ is a hyperdomain. Now, 

Let $0 \neq x \in A$.  Since $ \langle x^2 \rangle$ is a prime hyperideal of $A$ and $x^2 \subseteq  \langle x^2 \rangle$, we get $x \in \langle x^2 \rangle$ which means $x \in x^2 \circ a$ for some $a \in A$. It follows that $0 \in x^2 \circ a-x$. Since $0 \in x^2 \circ a-x \cap \langle 0 \rangle$, we get the result that $(x \circ a -1) \circ x  \subseteq x^2 \circ a-x \subseteq \langle 0 \rangle$. Let $t \in x \circ a$. Since $\langle 0 \rangle$ is a prime hyperideal, $ (t-1) \circ  x \subseteq \langle 0 \rangle$ and $x \neq 0$, we obtain $t-1=0$ and so $1 \in x \circ a$  which means $x$ is unit. Thus we conclude that $A$ is a hyperfield. 
\end{proof}
\begin{theorem} \label{s11}
Let  every hyperideal of $A$ be a strong $\mathcal{C}$-hyperideal, $J(A)=M$ and $v \in \mathbb{N}$. Then  every proper hyperideal in $A$ is a $(v+1,v)$-absorbing prime hyperideal if and only if  $A$ is local and $M^v=0$.
\end{theorem}

\begin{proof}
($\Longrightarrow$)  Suppose on contrary that $A$ is not local. By Corollary \ref{s4}, we conclude that every proper  hyperideal in $A$ is prime. Therefore, by Lemma \ref{s10},  we get the result that $A$ is a hyperfield which is impossible. Then $A$ is local. Let $M$ be the only maximal hyperideal of $A$. Since $\langle 0 \rangle$  is a $(v+1,v)$-absorbing prime hyperideal and $A$ is local, we have $M^v=0$ by the hypothesis and Proposition \ref{s5}. 

($\Longleftarrow$) Let $Q$ be an arbitary hyperideal and  $x_1 \circ \cdots \circ x_{v+1} \subseteq Q$ for $x_1,\cdots, x_{v+1} \in A \backslash U(A)$. Since $A$ is local, $M$ is the only maximal hyperideal of $A$ and  $x_1,\cdots, x_v \in A \backslash U(A)$, we conclude that $x_1 \circ \cdots \circ x_v \subseteq M^v=0$ which means $x_1 \circ \cdots \circ x_v \subseteq Q$, as required.
\end{proof}
\begin{theorem} \label{s8}
Let $Q$ be a proper hyperideal of $A$ and $u,v \in \mathbb{N}$ with $u >v$. Then $Q$ is a $(u,v)$-absorbing prime hyperideal of $A$ if and only if for proper hyperideals $Q_1,\cdots,Q_u$ of $A$, $Q_1 \circ \cdots \circ Q_u \subseteq Q$ implies that $Q_1 \circ \cdots \circ Q_v \subseteq Q$ or $Q_{v+1} \circ \cdots \circ Q_u \subseteq Q$. 
\end{theorem}
\begin{proof}
($\Longrightarrow$) Let $Q$ be a $(u,v)$-absorbing prime hyperideal of $A$. Assume that $Q_1 \circ \cdots \circ Q_u \subseteq Q$ where $Q_1,\cdots,Q_u$ are proper hyperideals of $A$ such that $Q_{v+1},\circ \cdots \circ Q_u \nsubseteq Q$. Then there exists $x_i \in Q_i$ for each $i \in \{v+1,\cdots,u\}$ such that $x_{v+1} \circ \cdots \circ x_u \nsubseteq Q$. Take any $y_i \in Q_i$ for each $i \in \{1,\cdots,v\}$. Therefore we have $y_1 \circ \cdots y_v \circ x_{v+1} \circ \cdots \circ x_u \subseteq Q$. This implies that $y_1 \circ \cdots y_v \subseteq Q$ as $Q$ is  a $(u,v)$-absorbing prime hyperideal of $A$ and $x_{v+1} \circ \cdots \circ x_u \nsubseteq Q$. Hence $Q_1 \circ \cdots \circ Q_v \subseteq Q$, as needed.

($\Longleftarrow$) Let $x_1 \circ \cdots \circ x_u \subseteq Q$ for some $x_1,\cdots,x_u \in A \backslash U(A)$ such that $x_1 \circ \cdots \circ x_v \nsubseteq Q$. Put $Q_i = \langle x_i \rangle$ for each $i \in \{1,\cdots,u\}$. Hence we have $Q_1 \circ \cdots \circ Q_u \subseteq Q$ and $Q_1 \circ \cdots \circ Q_v \nsubseteq Q$. By the hypothesis, we conclude that $Q_{v+1} \circ \cdots \circ Q_u \subseteq Q$ which implies $x_{v+1} \circ \cdots \circ x_u \subseteq Q$.  This means that $Q$ is a $(u,v)$-absorbing prime hyperideal of $A$.
\end{proof}
\begin{proposition} \label{s9}
Let  $u,v \in \mathbb{N}$ with $u >v$,    $Q$  a $(u,v)$-absorbing prime hyperideal of  $A$ and $x \in A \backslash (Q \cup U(A))$. Then $(Q : x)$ is a $(u-1,v-1)$-absorbing prime hyperideal of  $A$. 
\end{proposition}
\begin{proof}
Let $x_1 \circ \cdots \circ x_{u-1} \subseteq (Q : x)$ for $x_1, \cdots, x_{u-1} \in A \backslash U(A)$ such that $x_1 \circ \cdots \circ x_{v-1} \nsubseteq (Q : x)$. So we have $x \circ x_1 \circ \cdots \circ x_{u-1}  \subseteq Q$. Since $Q$ is  a $(u,v)$-absorbing prime hyperideal of  $A$ and $x \circ x_1 \circ \cdots \circ x_{v-1}  \nsubseteq Q$, we get the result that $x_v \circ \cdots x_{u-1} \ \subseteq Q \subseteq (Q : x)$.  This shows that $(Q : x)$ is a $(u-1,v-1)$-absorbing prime hyperideal of  $A$. 
\end{proof}
%\begin{theorem}
%Assume that $A$ is a hyperring such that the set of all hyperideals of $A$ are %linearly ordered by inclusion. If $Q$ is a $(v+1,v)$-absorbing prime strong  $\mathcal{C}$-hyperideal of $A$ that is not prime where $v \geq 2$, then $Q=M^n$ for the maximal hyperideal $M$ of $A$ and some $n \in \mathbb{N}$ with $2 \leq n \leq v$ and . 
%\end{theorem}
%\begin{proof}
%Suppose that$Q$ is a $(v+1,v)$-absorbing prime hyperideal of $A$ that is not prime. By Proposition \ref{s5}, we conclude that $A$ is local and $M^v \subseteq Q$ for the only maximal hyperideal $M$ of $A$. Put $n=\min \{t \ \vert \ M^t \subseteq Q \}$. If  $M^n \subsetneq Q$,   we have $x \in M^{n-1} \backslash Q$ and $y \in Q \backslash M^n$. By the hypothesis, we get the result that $y  \in \langle x \rangle$. Then there exists $ a \in A$ such that $y \in a \circ x$. If $a \notin M$, then we have $x \in e \circ x \subseteq a^{-1} \circ a \circ x \
%\end{proof}

Assume that $A$ is a multiplicative hyperring. Then the set of all hypermatrices of $A$ is denoted by  $M_m(A)$. Let  $I = (I_{ij})_{m \times m}, J = (J_{ij})_{m \times m} \in P^\star (M_m(A))$. Then $I \subseteq J$ if and only if $I_{ij} \subseteq J_{ij}$\cite{ameri}. 
\begin{theorem} \label{8} 
Let  $u,v \in \mathbb{N}$ with $u >v$ and  $Q$ be a hyperideal of $A$. If $M_m(Q)$ is a $(u,v)$-absorbing prime hyperideal of $M_m(A)$, then $Q$ is a $(u,v)$-absorbing prime hyperideal of $A$. 
\end{theorem}
\begin{proof}
Let $x_1 \circ \cdots  \circ x_u \subseteq Q$ for  $x_1, \cdots , x_u  \in A \backslash U(A)$. Therefore we have
\[\begin{pmatrix}
x_1 \circ \cdots  \circ x_u & 0 & \cdots & 0 \\
0 & 0 & \cdots & 0 \\
\vdots & \vdots & \ddots \vdots \\
0 & 0 & \cdots & 0 
\end{pmatrix}
\subseteq M_m(Q).\]
Since   $M_m(Q)$ is a $(u,v)$-absorbing prime hyperideal of $M_m(A)$ and 
\[ \begin{pmatrix}
x_1 \circ \cdots  \circ x_u & 0 & \cdots & 0\\
0 & 0 & \cdots & 0\\
\vdots & \vdots & \ddots \vdots\\
0 & 0 & \cdots & 0
\end{pmatrix}
=
\begin{pmatrix}
x_1 & 0 & \cdots & 0\\
0 & 0 & \cdots & 0\\
\vdots & \vdots & \ddots \vdots\\
0 & 0 & \cdots & 0
\end{pmatrix}
\circ \cdots \circ
\begin{pmatrix}
x_u & 0 & \cdots & 0\\
0 & 0 & \cdots & 0\\
\vdots & \vdots & \ddots \vdots\\
0 & 0 & \cdots & 0
\end{pmatrix}
\]
we conclude that 
\[ \begin{pmatrix}
x_1 & 0 & \cdots & 0\\
0 & 0 & \cdots & 0\\
\vdots& \vdots & \ddots \vdots\\
0 & 0 & \cdots & 0
\end{pmatrix} 
\circ \cdots \circ
\begin{pmatrix}
x_v & 0 & \cdots & 0\\
0 & 0 & \cdots & 0\\
\vdots& \vdots & \ddots \vdots\\
0 & 0 & \cdots & 0
\end{pmatrix}=
\begin{pmatrix}
x_1 \circ \cdots  \circ x_v & 0 & \cdots & 0\\
0 & 0 & \cdots & 0\\
\vdots& \vdots & \ddots \vdots\\
0 & 0 & \cdots & 0
\end{pmatrix}
\subseteq M_m(Q)\]\\
or 
\[ \begin{pmatrix}
x_{v+1} & 0 & \cdots & 0\\
0 & 0 & \cdots & 0\\
\vdots& \vdots & \ddots \vdots\\
0 & 0 & \cdots & 0
\end{pmatrix} 
\circ \cdots \circ
\begin{pmatrix}
x_u & 0 & \cdots & 0\\
0 & 0 & \cdots & 0\\
\vdots& \vdots & \ddots \vdots\\
0 & 0 & \cdots & 0
\end{pmatrix}=
\begin{pmatrix}
x_{v+1} \circ \cdots  \circ x_u & 0 & \cdots & 0\\
0 & 0 & \cdots & 0\\
\vdots& \vdots & \ddots \vdots\\
0 & 0 & \cdots & 0
\end{pmatrix}
\subseteq M_m(Q).\]

Thus we get the result that  $x_1 \circ \cdots  \circ x_v \subseteq Q$ or $x_{v+1} \circ \cdots  \circ x_u \subseteq Q$. Hence   $Q$ is a $(u,v)$-absorbing prime hyperideal of $A$. 
\end{proof}
Suppose that $(A,+,\circ)$ is a commutative multiplicative hyperring and $x$ is an indeterminate. Then $(A[x],+,\diamond)$ is a polynomail multiplicative hyperring where $tx^n \diamond sx^m=(t \circ s )x^{n+m}$ \cite{Ciampi}.
\begin{theorem} \label{polynomail}
Let  $u,v \in \mathbb{N}$ with $u >v$. If $Q$ is a  $(u,v)$-absorbing prime hyperideal of $(A,+,\circ)$, then $Q[x]$ is a $(u,v)$-absorbing prime hyperideal of $(A[x],+,\diamond)$. 
\end{theorem}
\begin{proof}
Assume that $t_1(x) \diamond \cdots \diamond t_u(x) \subseteq Q[x]$. Without
loss of generality, we may assume that $t_i(x)=a_ix^{n_i}$ for each $i \in \{1,\cdots, u \}$ and $a_i\in A$. Therefore $a_1 \circ \cdots \circ a_u x^{n_1+\cdots+n_u} \subseteq Q[x]$. This implies that $a_1 \circ \cdots \circ a_u \subseteq Q$. Since $Q$ is a  $(u,v)$-absorbing prime hyperideal of $(A,+,\circ)$, we obtain  $a_1 \circ \cdots \circ a_v \subseteq Q$ or $a_{v+1} \circ \cdots \circ a_u \subseteq Q$ which implies $t_1(x) \diamond \cdots \diamond t_v(x)=a_1x^{n_1} \diamond \cdots \diamond a_vx^{n_v}=a_1 \circ \cdots \circ a_vx^{n_1+\cdots+n_v} \subseteq Q[x]$ or $t_{v+1}(x) \diamond \cdots \diamond t_u(x)=a_{v+1}x^{n_1} \diamond \cdots \diamond a_ux^{n_u}=a_{v+1} \circ \cdots \circ a_ux^{n_{v+1}+\cdots+n_u} \subseteq Q[x]$. Thus $Q[x]$ is a $(u,v)$-absorbing prime hyperideal of $(A[x],+,\diamond)$. 
\end{proof}
In view of Theorem \ref{polynomail}, we get the following result.
\begin{corollary}
Suppose that $Q$ is a $(u,v)$-absorbing prime hyperideal of $A$. Then $Q[x]$ is a $(u,v)$-absorbing prime hyperideal of $A[x]$. 
\end{corollary}
 Recall from \cite{f10} that a mapping $\theta$ from the multiplicative hyperring
$(A_1, +_1, \circ _1)$ into the multiplicative hyperring $(A_2, +_2, \circ _2)$  is  a hyperring good homomorphism if $\theta(x +_1 y) =\theta(x)+_2 \theta(y)$ and $\theta(x\circ_1y) = \theta(x)\circ_2 \theta(y)$  for all $x,y \in A_1$.

\begin{theorem} \label{homo} 
Assume that $A_1$ and $A_2$ are two multiplicative hyperrins such that $\theta: A_1 \longrightarrow A_2$ is  a hyperring
good homomorphism and $\theta(x) \notin U(A_2)$ for every $x \in A_1 \backslash U(A_1)$. Let  $u,v \in \mathbb{N}$ with $u >v$. Then the following  hold. 
\begin{itemize}
\item[\rm{(i)}]~ If $Q_2$ is a $(u,v)$-absorbing prime hyperideal of $A_2$, then $\theta^{-1}(Q_2)$ is a $(u,v)$-absorbing prime hyperideal of $A_1$.
\item[\rm{(ii)}]~ If $\theta$ is surjective and $Q_1$ is a is a $(u,v)$-absorbing prime $\mathcal{C}$-hyperideal of $A_1$ with $Ker (\theta) \subseteq Q_1$, then $\theta(Q_1)$ is a $(u,v)$-absorbing prime hyperideal of $A_2$.
\end{itemize}
\end{theorem}
\begin{proof}
(i) Let $ x_1 \circ_1 \cdots \circ_1 x_u \subseteq \theta^{-1}(Q_2)$ for some $x_1,\cdots,x_u \in A_1$. Then we get  $ \theta(x_1 \circ_1 \cdots \circ_1 x_u)=\theta(x_1)  \circ_2 \cdots \circ_2 \theta(x_u) \subseteq Q_2$ because $\theta$ is  a good homomorphism. Since $Q_2$ is a $(u,v)$-absorbing prime hyperideal of $A_2$, we obtain $\theta(x_1 \circ_1 \cdots \circ_1 x_v)=\theta(x_1) \circ_2 \cdots \circ_2 \theta(x_v) \subseteq Q_2$   which means $x_1 \circ_1 \cdots \circ_1 x_v \subseteq \theta^{-1}(Q_2)$ or $\theta(x_{v+1} \circ_1 \cdots \circ_1 x_u)=\theta(x_{v+1}) \circ_2 \cdots \circ_2 \theta(x_u) \subseteq Q_2$ which implies $x_{v+1} \circ_1 \cdots \circ_1 x_u \subseteq \theta^{-1}(Q_2)$. Thus  $\theta^{-1}(Q_2)$ is a $(u,v)$-absorbing prime hyperideal of $A_1$.

(ii)  Let $ y_1 \circ_2 \cdots \circ_2 y_u \subseteq \theta(Q_1)$ for  $y_1, \cdots, y_u \in A_2$. Then  there exists $x_i \in A_1$  such that  $\theta(x_i)=y_i$ for each $i \in \{1,\cdots,u\}$ because $\theta$ is surjective. Therefore $\theta(x_1 \circ_1 \cdots \circ_1 x_u)=\theta(x_1) \circ_2 \cdots \circ_2 \theta(x_u)\subseteq \theta(Q_1)$. Now, pick any $t \in x_1 \circ_1 \cdots \circ_1 x_u$. Then $\theta(t) \in \theta(x_1 \circ_1 \cdots \circ_1 x_u) \subseteq \theta(Q_1)$ and so there exists $s \in Q_1$ such that $\theta(t)=\theta(s)$. Then we get $\theta(t-s)=0$ which means $t-s \in Ker (\theta)\subseteq Q_1$ and so  $t \in Q_1$. Therefore $x_1 \circ_1 \cdots \circ_1 x_u  \subseteq Q_1$ as $Q_1$ is a $\mathcal{C}$-hyperideal. Since  $Q_1$ is a $(u,v)$-absorbing prime hyperideal of $A_1$,  we get the result that  $x_1 \circ_1 \cdots \circ_1 x_v \subseteq Q_1$ or $x_{v+1} \circ_1 \cdots \circ_1 x_u \subseteq Q_1$. This means $y_1 \circ_2 \cdots \circ_2 y_v =\theta(x_1 \circ_1 \cdots \circ_1 x_v )\subseteq \theta(Q_1)$ or $y_{v+1} \circ_2 \cdots \circ_2 y_u=\theta(x_{v+1} \circ_1 \cdots \circ_1 x_u )\subseteq \theta(Q_1)$.  Consequently,  $\theta(Q_1)$ is a $(u,v)$-absorbing prime hyperideal of $A_2$.
\end{proof}
%Now, the following result obtained by the previous theorem directly.
Now, we have the following result.
\begin{corollary}
 Let the hyperideal $Q_1$ of $A$ be subset the $\mathcal{C}$-hyperideal  $Q_2$  of $A$, $x+Q_1 \notin U(A/Q_1)$ for all $x \in A \backslash U(A)$ and $u,v \in \mathbb{N}$ with $u>v$. Then  $Q_2$ is a $(u,v)$-absorbing prime hyperideal of $A$ if and only if   $Q_2/Q_1$ is a $(u,v)$-absorbing prime hyperideal of $A/Q_1$.
\end{corollary}
\begin{proof}
Consider the homomorphism $\pi :A \longrightarrow A/Q_1$ defined by $\pi(a)=a+Q_1$. Since $\pi$ is a good epimorphism,  the claim follows from Theorem \ref{homo} 
\end{proof}
Let $S$ be a non-empty subset of $A$.  If  $S$ is closed under the hypermultiplication and $S$ contains $1$, then $S$ refers to a multiplicative closed subset (briefly, MCS)\cite{ameri}. Consider the set $(A \times S / \sim)$ of equivalence classes denoted by  $S^{-1}A$  such that $(x_1,r_1) \sim (x_2,r_2)$ if and only if there exists $ r \in S $ with  $ r \circ r_1 \circ x_2=r \circ r_2 \circ x_1$.
The equivalence class of $(x,r) \in A \times S$ is denoted by $\frac{x}{r}$. The triple $(S^{-1}A, \oplus, \odot)$ is a multiplicative hyperring where

$\hspace{1cm}\frac{x_1}{r_1} \oplus \frac{x_2}{r_2}=\frac{r_1 \circ x_2+r_2 \circ x_1}{r_1 \circ r_2}=\{\frac{a+b}{c} \ \vert \ a \in r_1 \circ x_2 , b \in r_2 \circ x_1, c \in r_1 \circ r_2\}$

$\hspace{1cm}\frac{x_1}{r_1} \odot \frac{x_2}{r_2}=\frac{x_1 \circ a_2}{r_1 \circ r_2}=\{\frac{a}{b} \ \vert \ a \in x_1 \circ x_2, b \in r_1 \circ r_2\}$

The localization map $\pi: A \longrightarrow S^{-1}A$, defined by $a \mapsto \frac{a}{1}$,  is a homomorphism of hyperrings. Moreover, if $I$ is a hyperideal of $A$, then $S^{-1}I$ is a hyperideal of $S^{-1}A$ \cite{Mena}.
\begin{theorem}
Assume that $Q$ is a $\mathcal{C}$-hyperideal of $A$, $S$  a  MCS such that  $Q \cap S = \varnothing$ and $u,v \in \mathbb{N}$ with $u >v$. If $Q$ is a $(u, v)$-absorbing prime hyperideal of $A$, then $S^{-1}Q$ is a $(u-1, v-1)$-absorbing prime hyperideal of $S^{-1}A$. 
\end{theorem}
\begin{proof}
Let $\frac{x_1}{r_1} \odot \cdots \odot \frac{x_{u-1}}{r_{u-1}}=\frac{x_1 \circ  \cdots \circ x_{u-1}}{r_1 \circ \cdots \circ r_{u-1}} \subseteq S^{-1}Q$ for  $x_1,  \cdots , x_{u-1} \in A \backslash U(A)$ and $r_1,\cdots,r_{u-1} \in S$ such that $\frac{x_1}{r_1} \odot \cdots \odot \frac{x_{v-1}}{r_{v-1}} \nsubseteq S^{-1}Q$. Take any $a \in x_1 \circ  \cdots \circ x_{u-1}$ and $r \in r_1 \circ \cdots \circ r_{u-1}$. Hence $\frac{a}{r} \in \frac{x_1 \circ  \cdots \circ x_{u-1}}{r_1 \circ \cdots \circ r_{u-1}}$ and so $\frac{a}{r}=\frac{a^{\prime}}{r^{\prime}}$ for some $a^{\prime} \in Q$ and $r^{\prime} \in S$. Then there exists $t \in S$ such that $t \circ a \circ r^{\prime}=t \circ a^{\prime} \circ r$. This means $t \circ a \circ r^{\prime} \subseteq Q$. Since $a \in x_1 \circ  \cdots \circ x_{u-1}$, we conclude that $t \circ a \circ r^{\prime} \subseteq t \circ x_1 \circ  \cdots \circ x_{u-1} \circ r^{\prime}$. Therefore we have  $t \circ x_1 \circ  \cdots \circ x_{u-1} \circ r^{\prime} \subseteq Q$ as $Q$ is a $\mathcal{C}$-hyperideal of $A$ and $t \circ x_1 \circ  \cdots \circ x_{u-1} \circ r^{\prime} \cap Q \neq \varnothing$. Take any $s \in t \circ r^{\prime}$. If $s \circ x_1 \circ  \cdots \circ x_{v-1} \subseteq Q$, then we get $\frac{x_1}{r_1} \odot \cdots \odot \frac{x_{v-1}}{r_{v-1}}=\frac{x_1 \circ \cdots \circ x_{v-1}}{r_1 \circ \cdots \circ r_{v-1}}= \frac{s \circ x_1 \circ \cdots \circ x_{v-1}}{s \circ r_1 \circ \cdots \circ r_{v-1}}\subseteq S^{-1}Q$ which is impossible. Since $Q$ is a $(u, v)$-absorbing prime hyperideal of $A$,  $s \circ x_1 \circ  \cdots \circ x_{u-1} \subseteq Q$ and $s \circ x_1 \circ  \cdots \circ x_{v-1} \nsubseteq Q$, we have $ x_v \circ  \cdots \circ x_{u-1} \subseteq Q$ which implies $\frac{x_v}{r_v} \odot \cdots \odot \frac{x_{u-1}}{r_{u-1}}=\frac{x_v \circ \cdots \circ x_{u-1}}{r_v \circ \cdots \circ r_{u-1}}  \subseteq S^{-1}Q$. This shows that $S^{-1}Q$ is a $(u-1, v-1)$-absorbing prime hyperideal of $S^{-1}A$. 
\end{proof}

%%%%%%%%%%%%%%%%%%%%%%
%%%%%%%%%%%%%%%%%%%%%5

%%%%%%%%%%%%%%%%%%%%%%%%%%%%%%%%%%%%%%%%%%%
%%%%%%%%%%%%%%%%%%%%%%%%%

\end{document}